\documentclass[a4paper]{amsart}

\usepackage{latexsym} 
\usepackage{mathtools} 
\usepackage{amsmath, amssymb} 
\usepackage{amsthm} 
\usepackage{enumitem}
\usepackage[all]{xy}
\usepackage{hyperref}

\theoremstyle{plain}
\newtheorem{thm}{Theorem}[section]
\newtheorem{lemma}[thm]{Lemma}

\newtheorem{proposition}[thm]{Proposition}
\newtheorem{corollary}[thm]{Corollary}
\newtheorem{question}[thm]{Question}

\theoremstyle{definition}
\newtheorem{remark}[thm]{Remark}
\newtheorem{definition}[thm]{Definition}
\newtheorem{example}[thm]{Example}

\mathtoolsset{showonlyrefs=true}
\numberwithin{equation}{section}

\title{Dynamics of quandle orders}

\author{Chihaya Jibiki}
\address{Department of Mathematics, School of Science, Institute of Science Tokyo, 2-12-1 Ookayama, Meguro-ku, Tokyo 152-8550 Japan}
\email{chihaya.j@gmail.com}

\date{}


\begin{document}
\maketitle

\begin{abstract}
We construct a continuous map from the space of orders on quandles to the space of quandle actions on one-manifolds, providing an answer to a question posed by Idrissa Ba and Mohamed Elhamdadi. As an application of this map, we characterize isolated orders on quandles in terms of strong rigidity. As another application, we prove that there is no isolated right order on the free quandles, except for specific cases.
\end{abstract}


\section{Introduction}

Quandles are algebraic structures whose axioms correspond to the Reidemeister moves in knot theory. Certain quandles, called keis, were first introduced by Mituhisa Takasaki \cite{Takasaki1943AbstractionOS} to find a distributive algebraic structure which captures the notion of reflection in the context of finite geometry, in comparison with the symmetric space in differential geometry. In the 1980's, Joyce \cite{MR638121} and Matveev \cite{MR672410} independently provided the modern definition of quandles and developed their theory, particularly in relation to knot theory. They introduced the fundamental quandle of a knot, as a complete invariant up to reversal and mirror image. That is, two knots $K$ and $K'$ are equivalent (up to reversal and mirror image) if and only if their fundamental quandles $Q(K)$ and $Q(K')$ are isomorphic as quandles. Since then, quandles have been studied from various topological and algebraic perspectives.

Given a set with a binary operation, right orders and circular orders can be defined on the set\footnote{Left orders can also be defined, but in this paper, we focus on right orders. Note that left orderability is not equivalent to right orderability on quandles.}. Orderability was first introduced in group theory by R. Dedekind and O. H\"{o}lder at the end of the 19th century and the beginning of 20th century, respectively. Dedekind characterized the real numbers as a complete ordered Abelian group, while H\"{o}lder proved that any Archimedean Abelian ordered group is order isomorphic to a subgroup of the additive real numbers with the standard order. Recently, the orderability of quandles has been well studied \cite{MR4330281,MR4450681,ba2023knot,ba2022circularorderabilityquandles}, revealing relationships between orderability and the fundamental quandles of knots.

The sets of all right orders and circular orders on a set $X$ form topological spaces $RO(X)$ and $RCO(X)$, respectively\footnote{In the quandle case, these spaces are denoted by $RQ(Q)$ and $RCQ(Q)$ for each quandle $Q$.}. Over the past few decades, several researchers have explored the properties of these spaces. Linnell \cite{MR2765563} studied their cardinality; Sikora \cite{MR2069015}, Dabkowska et al. \cite{MR2320157} and Ba and Elhamdadi \cite{ba2022circularorderabilityquandles} proved that they are compact; Ha and Harizanov \cite{MR3835743} investigated their algebraic and computability-theoretic properties. It follows that these spaces are Cantor sets unless they contain isolated points, called the isolated orders in both the right and circular order cases. Isolated orders have been studied in many groups; for example, braid groups \cite{MR1859702} and torus knot groups \cite{MR2745552} have isolated orders while virtually solvable groups and surface groups \cite{MR3460331} do not. These are just a few examples.

In the  case of groups, a powerful tool for studying the set of orders is dynamical realization, given by a continuous map $R:CO(G)\to \textrm{Hom}(G,\textrm{Homeo}(S^1))$ or $R:RO(G)\to \textrm{Hom}(G,\textrm{Homeo}(\mathbb{R}))$. In several works on isolated orders, dynamical realizations play a major role. Furthermore, they connect the theory of orders to various topological structures, including surface groups \cite{MR3664524}, triangle groups \cite{MR4584769}, mapping class groups \cite{MR4266358,MR1805402} and foliations \cite{MR1965363}.

In Section \ref{sec:3}, we construct a dynamical realization in the quandle case.

\begin{thm}\label{thm:main1}
    Let $Q$ be a countable, circular (resp. right) orderable quandle which is semi-latin at $\hat{q}\in Q$. There exist two continuous maps:
    \begin{align}
        &R:RCQ(Q)\to \textrm{Hom}(Q,\textrm{Conj}(\textrm{Homeo}_+(S^1)))\\
        &(\textrm{resp.}\,R:RQ(Q)\to \textrm{Hom}(Q,\textrm{Conj}(\textrm{Homeo}_+(\mathbb{R})))).
    \end{align}
     Moreover, both actions of each image of $R$ are faithful.
\end{thm}
This provides an answer to \cite[Question 4.3]{ba2022circularorderabilityquandles}.

In Section \ref{sec:4}, using the theorem above, we characterize isolated orders. This is a reformulation of Mann and Rivas' \cite[Theorem 1.2]{MR3887426} result for groups, applied to quandles. The corresponding result for right orders is given in Theorem \ref{thm:main2.5}.

\begin{thm}\label{thm:main2}
    Let $G$ be a countable group, $A$ a subset of $G$ and $\mathcal{D}(A)$ the Dehn quandle which is semi-latin at $\hat{q}\in \mathcal{D}({}^GA)$. Let $c$ be a circular order on $\mathcal{D}({}^GA)$ whose dynamical realization $\rho$ can be extended to a group action of $G$.
    
    Then, $c$ is isolated if and only if $\rho$ is strongly rigid in the following sense: for every $\rho'$ sufficiently close to $\rho$ in $\textrm{Hom}(\mathcal{D}({}^GA),\textrm{Conj}(\textrm{Homeo}_+(S^1)))$ there exists a continuous, degree-one monotone map $h:S^1\to S^1$ fixing the basepoint $x_0$ of the realization for $\hat{q}$, and such that $\rho(q)\circ h=h\circ \rho'(q)$ for all $q\in \mathcal{D}({}^GA)$.
\end{thm}
In particular, we provide a complete characterization of isolated orders on the Dehn quandles generated by free groups.

\begin{corollary}\label{cor:main2}
    Let $G$ be a countable group, $F(X)$ the free group generated by $X$ and $F(X)\star G$ their free product. Then, a circular (resp. right) order on the Dehn quandle $\mathcal{D}({}^{F(X)\star G}X)$ is isolated if and only if its dynamical realization $\rho$ is strongly rigid.
\end{corollary}
However, Theorem \ref{thm:main2} assumes specific conditions for quandles.

\begin{question}
    Is there a complete characterization of isolated orders on quandles which are not either Dehn quandles or semi-latin, based on quandle actions? 
\end{question}

In terms of sets having the universal property of freeness, it is known that free groups do not admit any isolated circular order in the odd-rank case \cite{MR4033501}, any isolated right orders \cite{MR787955} or any isolated bi-orders \cite{MR4584679} in any rank. However, free groups do admit isolated circular orders in the even-rank case. In the quandle case, it is known that the conjugation quandle of the countably infinite rank free group does not admit any isolated right order \cite{MR3835743}. In Section \ref{sec:5}, we show that no order is isolated on free quandles\footnote{Note that any free quandle is a subquandle of some conjugation quandle of a free group.}, except for two specific cases.

\begin{thm}\label{thm:main3}
    Let $n\geq2$, $X=\{x_1,\dots,x_n\}$ be a set and $FQ(X)$ the free quandle generated by $X$. Then, any right order on $FQ(X)$ is non-isolated order, except for the following two types of right orders $<\in RQ(FQ(X))$.
    \begin{enumerate}[label=\textbf{E\arabic*}]
    \item For any $\hat{q}\in FQ(X)$, there exist elements $M,m\in FQ(X)$, for any $q\in FQ(X)$ such that $\hat{q}*m<\hat{q}*q<\hat{q}*M$.
    \item For any $\hat{q}\in FQ(X)$, there exist $M,m\in FQ(X)$, for any ${}^Ux_i,{}^Vx_j\in FQ(X)$, such that
    \begin{align}
        &\hat{q}*M<\hat{q}*{}^Ux_i<\hat{q}*{}^Vx_j \textrm{\,\,implies\,\,}i=j,\textrm{\, and}\\
        &\hat{q}*{}^Ux_i<\hat{q}*{}^Vx_j<\hat{q}*m \textrm{\,\,implies\,\,}i=j.
    \end{align}
    \end{enumerate}
\end{thm}

This rises the following question.

\begin{question}
    Are there no isolated right order on any free quandle?
\end{question}

\section*{Acknowledgements}
The author is supported by JST SPRING, Japan Grant Number JPMJSP2106.


\section{Backgrounds}\label{sec:backgrounds}

\subsection{Definition of quandles}

\begin{definition}
    A {\it quandle} is a non-empty set $Q$ with a binary operation $*$ satisfying the following axioms:
    \begin{enumerate}[label=\textrm{Q}\arabic*]
    \item $q*q=q$ for all $q\in Q$.
    \item For each $q,r\in Q$, there exists a unique $s\in Q$ such that $q=s*r$.
    \item $(q*r)*s=(q*s)*(r*s)$ for all $q,r,s\in Q$.
    \end{enumerate}
    Consider the right multiplication $S_q:Q\to Q$ defined by $r\mapsto r*q$. Then axioms Q$2$ and Q$3$ state that $S_q$ is an automorphism of the quandle $Q$, and axiom Q$1$ states that $q$ is a fixed point of $S_q$. This gives a dual binary operation $*^{-1}$ on $Q$ defined by $q*^{-1}r=s$ if $q=s*r$. Thus, axiom Q$2$ is equivalent to
    \[
        (q*r)*^{-1}r=q=(q*^{-1}r)*r
    \]
    for all $q,r\in Q$, and hence it allows us cancellation from right. Axioms Q$1$ and Q$3$ are referred to as the idempotency and distributivity axioms, respectively. Idempotency and cancellation together give $q*^{-1}q=q$ for all $q\in Q$.
\end{definition}

In what follows, we write $(q*r)*s$ as $q*r*s$ for simplicity.

\begin{example}
    Here are some examples of quandles.
    \begin{itemize}
        \item Every set with the binary operation $q*r=q$ is a quandle called the {\it trivial quandle}.
        \item Let $G$ be a group. For $n\in \mathbb{Z}$, the binary operation $q*r=r^nqr^{-n}$ turns $G$ into the quandle $\textrm{Conj}_n(G)$ called the {\it $n$-conjugation quandle} of $G$. For $n = 1$, this is simply denoted by $\textrm{Conj(G)}$.
        \item Let $G$ be a group. For $\varphi\in \textrm{Aut}(G)$, the binary operation $q*r=\varphi(qr^{-1})r$ forms the quandle $\textrm{Alex}(G,\varphi)$ referred to as the {\it generalized Alexander quandle} of $G$ with respect to $\varphi$. In particular, if let $G$ be the additive group $\mathbb{R}$ and $\varphi$ a $t$-multiplication by some $t\in \mathbb{R}\setminus\{0\}$, then the binary operation is given by $q*r=tq+(1-t)r$  called simply the {\it Alexander quandle} $\textrm{Alex}(\mathbb{R},t)$.
        \item Let $G$ be a group. The binary operation $q*r=rq^{-1}r$ turns $G$ into the quandle $\textrm{Core}(G)$ called the {\it core quandle} of $G$.
        \item Every link can be assigned a quandle called the {\it link quandle}, which is a complete invariant of non-split links up to weak equivalence. Other topological quandles can be seen here \cite{MR3729413}.
        \item Let $X$ be a set and $F(X)$ the free group on $X$. Define
        \[
            S\times F(X)=\{{}^wq\mid q\in X,w\in F(X)\}
        \]
        and for $q,r\in X$ and $v,w\in F(X)$ the binary operation 
        \[
            {}^wq*{}^vr= {}^{wqw^{-1}v}r.
        \]
        We define the {\it free quandle} $FQ(X)$ on $X$ as the quotient of $S\times F(X)$ by the equivalence relation generated by
        \[
            {}^wq={}^{wq}q.
        \]
        The quandle $FQ(X)$ satisfies the universal freeness property similar to that of free groups, and is a subquandle of $\textrm{Conj(F(X))}$(see \cite{MR4669143}).
        \item Let $G$ be a group, $A$ a non-empty subset of $G$ and $\mathcal{D}({}^GA)$ the set of all conjugates of elements of $A$ in $G$. The binary operation $q*r=rqr^{-1}$ turns $\mathcal{D}({}^GA)$ into the quandle $\mathcal{D}({}^GA)$ called the {\it Dehn quandle} of $G$ with respect to $A$. It is known by Dhanwani, Raundal and Singh \cite{MR4669143} that all of core quandles, generalized Alexander quandles and subquandles of any $n$-conjugation quandles (including any free quandles) are Dehn quandles.
    \end{itemize}
\end{example}

\begin{definition}
    We see several notions regarding quandles.
    \begin{itemize}
        \item A quandle $Q$ is said to be {\it latin} at $\hat{q}$ if the map $L_{\hat{q}}:Q\to Q$ defined by $L_{\hat{q}}(r)=\hat{q}*r$ is a bijection for any $r\in Q$. A quandle is said to be {\it semi-latin} if $L_{\hat{q}}$ is injective.
        \item A quandle is {\it involutory} if $q*r=q*^{-1}r$ for all $q,r\in Q$(Involutory quandles were called {\it kei} and studied by Takasaki \cite{Takasaki1943AbstractionOS} at first).
        \item A {\it subquandle} of a quandle $Q$ is a subset of $Q$ which is closed under the quandle operation.
        \item A {\it homomorphism of quandles} $P$ and $Q$ is a map $\varphi:P\to Q$ with $\varphi(q*r)=\varphi(q)*\varphi(r)$ for all $q,r\in P$. By cancellation in $P$ and $Q$, this also implies $\varphi(q*^{-1}r)=\varphi(q)*^{-1}\varphi(r)$ for all $q,r\in P$.
        \item A {\it left action of a quandle} $Q$ on a set $S$ is defined as a quandle homomorphism $\psi:Q\to \textrm{Conj}(\Sigma_S)$, where $\Sigma_S$ is the symmetric group of $S$.
        We call the set $\textrm{Stab}_{\psi}(s)=\{q\in Q \mid \psi(q)(s)=s\}$ the {\it stabilizer} of $s\in S$ under $\psi$.
    \end{itemize}
\end{definition}

\begin{proposition}
    Let $Q$ and $P$ be quandles and $\varphi:P\to Q$ a homomorphism. Then, $\textrm{Conj}(\cdot)$ is a functor from the category of quandles to that of groups.
\end{proposition}


\subsection{Definition of orders}

\begin{definition}
    Let $Q$ be a quandle. A {\it right circular order} is a map $c: Q^3\to \{-1,0,1\}$ satisfying the following properties.
    \begin{enumerate}[label=\textbf{(\roman*)}]
    \item $c(q_1,q_2,q_3)=0$ if and only if $i=j$ for some $i\neq j$.
    \item $c(q_2,q_3,a_4)-c(q_1,q_3,a_4)+c(q_1,q_2,a_4)-c(q_1,q_2,a_3)=0$ for any $q_1,q_2,q_3,q_4\in Q$.
    \item $c(q_1,q_2,a_3)=c(q_1*q,q_2*q,q_3*q)$ for any $q,q_1,q_2,q_3\in Q$.
    \end{enumerate}
    Left circular orders are defined similarly by replacing the condition (iii) with the following:
    \[
        \textrm{(iii)'\,}c(q_1,q_2,a_3)=c(q*q_1,q*q_2,q*q_3) \textrm{\,for\,any\,} q,q_1,q_2,q_3\in Q.
    \]
\end{definition}
A right and left circular order is called {\it bi-circular order}. If $Q$ has right (resp. left or bi-) circular orders, then $Q$ is said to be {\it right (resp. left or bi-) circular orderable}. 
\begin{definition}
    Let $Q$ be a quandle. A {\it right (resp. left) total order $<$} is a total order which satisfies that
    \begin{align}
        &q_1<q_2 \textrm{\,\,implies\,\,}q_1*q<q_2*q \textrm{\, for\, any\,}q,q_1q_2\in Q\\
        &(\textrm{resp.\,} q_1<q_2 \textrm{\,implies\,}q*q_1<q*q_2 \textrm{\, for\, any\,}q,q_1q_2\in Q).
    \end{align}
    The right and left total order is called {\it bi-order}. If the quandle $Q$ has a right (resp. left or bi-) total order, then $Q$ is said to be {\it right (resp. left or bi-) total orderable}. 
\end{definition}

Since every set is a trivial quandle, the natural total order on $\mathbb{R}$ is a right total order denoted by $<_{R}$. Similarly the natural circular order on $S^1$ is a right circular order denoted by $\textrm{ord}:(S^1)^3\to \{-1,0,1\}$.

\begin{lemma}
    A right (resp. left) total orderable quandle is also right (resp. left) circular orderable.
\end{lemma}

\begin{proof}
    Let $<$ be a right total order on a quandle $Q$. Define the map $c_<:Q^3\to\{-1,0,1\}$ by
    \begin{align}
        c_<(q_1,q_2,q_3)=
        \begin{cases*}
            \textrm{sgn}(\sigma) & if $q_{\sigma(1)}<q_{\sigma(2)}<q_{\sigma(3)}$, \\
            0 & otherwise.
        \end{cases*}
    \end{align}
    By definition, $c_<$ is a right circular order.
\end{proof}

We recall several facts, focusing on right orderable and the property of latin or semi-latin.

\begin{proposition}\label{prop:facts}
    \begin{enumerate}[label=\textbf{(\roman*)}]
    \item Every trivial quandle is right total orderable but not semi-latin at any point.
    \item Every non-trivial finite quandle is not right circular orderable.
    \item If $G$ is a bi-total orderable group, then $\textrm{Conj(G)}$ is right total orderable.
    \item If $G$ is a group, then $\textrm{Core(G)}$ is a kei. Every non-trivial kei is not right circular orderable.
    \item If $t\in \mathbb{R}\setminus\{0\}$ is a positive number, then $\textrm{Alex}(\mathbb{R},t)$ is right total orderable. If $u\in \mathbb{R}\setminus\{0\}$ does not equal to one, then $\textrm{Alex}(\mathbb{R},u)$ is latin at all points.
    \item If $G$ is a right total orderable group and $\varphi\in \textrm{Aut}(G)$ is order-preserving, then $\textrm{Alex}(G,\varphi)$ is right total orderable. If $H$ is a group and $\psi\in \textrm{Aut}(H)$ is fixed point free, then $\textrm{Alex}(H,\psi)$ is semi-latin at all points. 
    \item Every free quandle is right total orderable and semi-latin at all points.
    \item If $Q$ is a quandle and $c$ is a right circular order on $Q$, then for each $q,q_1,q_2,q_3\in Q$, we have:
    \[
        c(q_1,q_2,q_3)=c(q_1*^{-1}q,q_2*^{-1}q,q_3*^{-1}q).
    \]
    If $<$ is a right total order on $Q$, then for each $q,q_1,q_2\in Q$, we have:
    \[
        q_1<q_2\,\textrm{\,implies\,}\,q_1*^{-1}q<q_2*^{-1}q.
    \]
    \end{enumerate}
\end{proposition}

\begin{proof}
    The proofs can be found in \cite{MR4450681,MR4330281,ba2022circularorderabilityquandles,ba2023knot}.
\end{proof}


\begin{definition}
    Let $Q$ be a quandle. The set of {\it right circular orders} on $Q$ is denoted by $RCQ(Q)$. Similarly, the set of {\it right total orders} on $Q$ is denoted by $RQ(Q)$.
\end{definition}

We topologize the set $RCQ(Q)$ as a subspace of $\{-1,0,1\}^{Q^3}$, the space
of all maps from $Q^3$ to $\{-1,0,1\}$ with the Tychonoff topology. We define
\begin{align}
    &\Gamma(Q)=\{s=(x,y,z)\in Q^3\mid x=y \textrm{\,or\,}y=z\textrm{\,or\,}z=x\}\,\textrm{and}\\
    &U_s=\{c\in RCQ(Q)\mid c(s)=1\}
\end{align}
for some $s\in Q^3\setminus\Gamma(Q)$.

\begin{lemma}[Ba and Elhamdadi, Lemma 4.1 \cite{ba2022circularorderabilityquandles}]
    The set $\{U_s\}_{s\in Q^3\setminus\Gamma(Q)}$ is a subbasis for the topology
on $RCQ(Q)$.
\end{lemma}

\begin{thm}[Ba and Elhamdadi, Theorem 4.1 \cite{ba2022circularorderabilityquandles}]\label{thm:cpt}
    The space $RCQ(Q)$ is compact.
\end{thm}

We topologize the set $RQ(Q)$ as a subspace of $RCQ(Q)$. There is the following fact about these spaces.

\begin{thm}[Dabkowska, Dabkowski, Harizanov, Przytycki, and Veve, Theorem 2 \cite{MR2320157}]\label{thm:Cantor}
    If $Q$ is a quandle, then $RQ(Q)$ is compact, totally disconnected. If $Q$ is countable, then the space is metrizable.
\end{thm}

One can prove the circular version of the above theorem to use Theorem \ref{thm:cpt}.

\begin{corollary}
    If $Q$ is a quandle, then $RCQ(Q)$ is compact, totally disconnected. If $Q$ is countable, then the space is metrizable.    
\end{corollary}

\begin{corollary}
    Let $Q$ be a countable quandle. If $RCQ(Q)$ has no isolated points, then $RCQ(Q)$ is homeomorphic to the Cantor set. Similarly,  if $RQ(Q)$ has no isolated points, then $RQ(Q)$ is homeomorphic to the Cantor set. 
\end{corollary}

\begin{definition}
    The isolated point in $RCQ(Q)$ (resp. $RQ(Q)$) are called a {\it isolated right circular (resp. total) order} on $Q$. 
\end{definition}

For simplicity, we will refer to right circular orders as circular orders and right total orders as right orders in what follows. 


\section{Dynamical Realizations of Quandles}\label{sec:3}

\begin{definition}
    A quandle $Q$ is said to be {\it weak-latin} if for any $q,r\in Q$, there exists an element $\hat{q}\in Q$ such that $\hat{q}*q=\hat{q}*r$ implies $q=r$.
\end{definition}

\begin{proposition}\label{prop:weak-latin}
    \begin{enumerate}[label=\textbf{(\arabic*)}]
    \item If a quandle is latin at some point, then it is semi-latin at that point.
    \item If a quandle is semi-latin at some point, then it is weak-latin. 
    \item Every subquandle of a weak-latin quandle is weak-latin.
    \item For any group $G$, $\textrm{Conj}(G)$ is weak-latin if and only if $G$ is centerless. In particular, if $G$ is centerless, then the Dehn quandle $\mathcal{D}({}^GA)$ is weak-latin  for any non-empty subset $A\subset G$. 
    \item Every non-trivial trivial quandle is not weak-latin.
    \end{enumerate}
\end{proposition}

\begin{proof}
    We prove (4).  Take $g\in G$ and $q,r\in \mathcal{D}({}^GA)$ satisfying that $g$ does not commute with $r^{-1}q$. Assume $g*q=g*r$. This implies that $(r^{-1}q)g(r^{-1}q)^{-1}=g$, which implies $q=r$.
 
    The converse can be proved similarly.
\end{proof}

\begin{lemma}\label{lem:const}
    Let $Q$ be a countable weak-latin quandle and $c$ a circular order on $Q$. Then, there is a quandle action on the circle as follows:
    \begin{align}
        \rho:Q\hookrightarrow \textrm{Conj}(\textrm{Homeo}_+(S^1)).
    \end{align}
\end{lemma}

\begin{proof}
    The proof is based on the method for the dynamical realization of groups. Let $\{q_i\}_{i\in \mathbb{N}}$ be an enumeration of the elements of $Q$. Define an order-preserving embedding $\iota:Q\hookrightarrow S^1$ as follows. Let $\iota(q_1)$ and $\iota(q_2)$ be arbitrary distinct points. Then, having embedded $q_1,\dots,q_{n-1}$, send $q_n$ to the midpoint of the unique connected component of $S^1\setminus \{\iota(q_1),\dots,\iota(q_{n-1})\}$ such that 
    \[
        c(q_i,q_j,q_k)=\textrm{ord}(\iota(q_i),\iota(q_j),\iota(q_k))
    \]
    holds for all $i,j,k\leq n$.

    Notice that $q$ acts naturally on $\iota(q')$ by $q\cdot \iota(q')=\iota(q'*q)$. Since $\iota$ is an order-preserving embedding, this right action extends continuously to the closure of the set $\iota(Q)$. We can extend the action to the circle by extending the map $q$ affinely to each interval of the complement of the closure $\overline{\iota(Q)}$. This gives the desired map $\rho:Q\to\textrm{Homeo}_+(S^1)$. 

    Similarly, define $\bar{\rho}:Q\to \textrm{Homeo}_+(S^1)$ by setting $q\cdot \iota(q')=\iota(q'*^{-1}q)$. The axiom Q$2$ implies that $\bar{\rho}(q)=\rho(q)^{-1}$ for every $q\in Q$. 
    
    For each $q,r,s\in Q$, we have the following:
    \begin{align}
        \rho(r*s)(\iota(q))&=\iota(q*(r*s))\\
        &=\iota(q*^{-1}s*r*s)\\
        &=\rho(s)\rho(r)\rho(s)^{-1}(\iota(q))\\
        &=\rho(r)*\rho(s)(\iota(q)),
    \end{align}
    which implies that $\rho$ is a quandle action. The weak-latin property shows that $\rho$ is injective.
\end{proof}

We call $\rho$ the {\it dynamical realization with $c$}.
The construction of $\rho$ depends on the choice of enumeration. However, it turn out to be well-defined up to topological conjugacy. 

\begin{lemma}\label{lem:well-def}
    The realization map $R:RCQ(Q)\to\textrm{Hom}( Q, \textrm{Conj}(\textrm{Homeo}_+(S^1)))/\sim$ is well-defined up to topological conjugacy.
\end{lemma}

\begin{proof}
    Let $\rho,\rho':Q\to \textrm{Homeo}_+(S^1)$ be constructed from the enumerations $\{q_i\}_{i\in \mathbb{N}}$ and $\{q_i'\}_{i\in \mathbb{N}}$ and embeddings $\iota,\iota'$, respectively. We define the following bijection:
    \begin{align}
        h:\iota'(Q)&\to \iota(Q)\\
        \iota'(q)&\mapsto\iota(q).
    \end{align}
    Note that $h$ is order-preserving. For each $q,r\in Q$, it follows that:
    \begin{align}
        h\rho'(r) h^{-1}(\iota(q))&=h\rho'(r)(\iota'(q))\\
        &=h(\iota'(q*r))\\
        &=\iota(q*r)\\
        &=\rho(r)(\iota(q)),
    \end{align}
    which means that $h\circ \rho'(r)\circ h^{-1}=\rho(q')$ on $\iota(Q)$. The map $h$ can be extended to the circle.
\end{proof}

We also refer to the topological conjugacy class as the dynamical realization with $c$.

Next, we also define the dynamical realizations of right orders.

\begin{lemma}\label{lem:const(total)}
    Let $Q$ be a countable quandle which is weak-latin and $<$ a right order on $Q$. Then, there is a quandle right action on the real line as follows:
    \begin{align}
        \rho:Q\hookrightarrow \textrm{Conj}(\textrm{Homeo}_+(\mathbb{R})).
    \end{align}
\end{lemma}

\begin{proof}
    Let $\{q_i\}_{i\in \mathbb{N}}$ be an enumeration of the elements of $Q$. Define an order-preserving embedding $\iota:Q\hookrightarrow \mathbb{R}$ as follows. Let $\iota(q_1)$ be an arbitrary point. Then, assume we have defined $\iota$ on the subset
    $\{q_1,\dots,q_{n-1}\}$ and let us define $\iota(q_{n})$. Order the subset $\{q_1,\dots,q_{n-1}\}$ as
    \[
        q_{i_1}<q_{i_2}<\cdots <q_{i_{n-1}}.
    \]

    If $q_n<q_{i_1}$, define $\iota(q_{n})=\iota(q_{i_1})-1$, if $q_{i_{n-1}}<q_{i_n}$, define $\iota(q_{n})=\iota(q_{n-1}) +1$, and if $q_{i_k}<q_{n}<q_{i_{k+1}}$, define $\iota(q_n)=(1/2)(\iota(q_{i_k})+\iota(q_{k_{k+1}}))$.

    Using the same approach as in the case of circular orders, we define the right action $\rho:Q\hookrightarrow \textrm{Conj}(\textrm{Homeo}_+(\mathbb{R}))$, which is injective.
\end{proof}

\begin{lemma}
    The realization map $R:RQ(Q)\to\textrm{Hom}( Q, \textrm{Conj}(\textrm{Homeo}_+(\mathbb{R})))/\sim$ is well-defined up to topological conjugacy.
\end{lemma}

\begin{proof}
    The proof is similar to that of Lemma \ref{lem:well-def}.
\end{proof}

\begin{proposition}
    Let $Q$ be a countable, circular or right orderable quandle which is semi-latin at $\hat{q}\in Q$ and $\{q_i\}_{i\in \mathbb{N}}$ an enumeration with $\hat{q}=q_1$. Then, the dynamical realization $\rho$ constructed from $\{q_i\}_{i\in \mathbb{N}}$ acts freely on $\iota(\hat{q})$.
\end{proposition}

\begin{proof}
    This is trivial since the map $L_{\hat{q}}$ is injective. 
\end{proof}

For any quandle $Q$, the space $\textrm{Hom}(Q,\textrm{Conj}(\textrm{Homeo}_+(S^1)))$ has a natural topology; a neighborhood basis of an action $\rho_0$ is given by the sets
\begin{align}
    O_{(F,\varepsilon)}(\rho_0)= \left\{ \rho\in\textrm{Hom}(Q,\textrm{Conj}(\textrm{Homeo}_+(S^1))) \middle|
    \begin{aligned}
        d(\rho_0(q)(x),\rho(q)(x))<\varepsilon\\
        \textrm{for\,all\,}x\in S^1,q\in F
    \end{aligned}
    \right\}
\end{align}
where $F$ ranges over all finite subsets of $Q$, and $d$ is the standard length metric on $S^1$. By fixing some point $p\in S^1$, the space  $\textrm{Hom}(Q,\textrm{Conj}(\textrm{Homeo}_+(\mathbb{R})))$ of actions of
$Q$ on $\mathbb{R}$ can be identified with the closed subset
\[
    \{\rho\in\textrm{Hom}(Q,\textrm{Conj}(\textrm{Homeo}_+(S^1)))\mid \rho(q)(p)=p \textrm{\, for\,all\,}q\in Q\}
\]
and its usual (compact–open) topology is just the subset topology.

In what follows, given the countable quandle $Q$ which is semi-latin at $\hat{q}\in Q$ and a circular or right order on $Q$, we construct its dynamical realization $\rho$ with $q_1=\hat{q}$ and $\iota(\hat{q})=x_0\in S^1$, denoting $x_0$ as the {\it basepoint}.

\begin{proof}[Proof of Theorem \ref{thm:main1}]
    We prove this similarly to \cite[Proposition 3.3]{MR3887426}.

    Let $c_0\in RCQ(Q)$, and let $\rho_0$ be a dynamical realization of $c_0$. Given a neighborhood $O_{F,\varepsilon}(\rho_0)$ of $\rho_0$ in $\textrm{Hom}(Q,\textrm{Conj}(\textrm{Homeo}_+(S^1)))$. Let $S\subset Q$ be a finite set with $F\subset S$ and $\#S>1/\varepsilon$.

    After conjugating $\rho_0$, we assume that $\rho(S)(x_0)$ partitions $S^1$ into intervals of equal length, each of length less than $\varepsilon$. We now show that every circular order $c$ that agrees with $c_0$ on the finite set $\hat{q}*S*S \cup \hat{q}*S*^{-1}S$ has a conjugate of a dynamical realization in the neighborhood $O_{S,\varepsilon}(\rho_0)\subset O_{F,\varepsilon}(\rho_0)$. Given such a circular order $c$, let $\rho$ be a dynamical realization of $c$ such that $\rho(s')^{\pm1}\rho(s)(x_0)=\rho_0(s')^{\pm1}\rho_0(s)(x_0)$ for all $s,s'\in S$. Fix $s\in S$. By construction $\rho(s)$ and $\rho_0(s)$ agree on every point of $\rho_0(S)(x_0)$.
    
    Let $I$ be any connected component of $S^1\setminus \rho_0(S)(x_0)$ and $x\in I$. If $\rho(s)(I)$ has length at most $\varepsilon$, then $\rho(s)(x)$ and $\rho_0(s)(x)$ differ by a distance of at most $\varepsilon$. If $\rho(s)(I)$ has length greater than $\varepsilon$, consider instead the partition of $\rho_0(s)(I)$ by $\rho(S)(x_0)\cap \rho_0(s)(I)$, this is a partition into intervals of length less than $\varepsilon$. Considering images of $\bar{\rho}(s)$ shows that $\rho(s)(x)$ and $\rho_0(s)(x)$ must lie in the same subinterval of the partition, and hence differ by a distance of at most $\varepsilon$.

    Since any right order is a circular order, Lemma \ref{lem:conti} shows that Theorem \ref{thm:main1} in the right order case is also proved in the same manner as above.
\end{proof} 

\subsection{Conjugacy-invariant orders}

We introduce the case where the quandle action of a dynamical realization is obtained from a group action.

\begin{definition}
    Let $G$ be a group. We define the {\it conjugation-invariant circular order} $c_{ci}$ to be a map $c_{ci}:G^3\to\{-1,0,1\}$ satisfying the following conditions.
    
    \begin{enumerate}[label=\textbf{(\roman*)}]
    \item $c_{ci}(g_1,g_2,g_3)=0$ if and only if $i=j$ for some $i\neq j$.
    \item $c_{ci}(g_2,g_3,a_4)-c_{ci}(g_1,g_3,a_4)+c_{ci}(g_1,g_2,a_4)-c_{ci}(g_1,g_2,a_3)=0$ for any $g_1,g_2,g_3,g_4\in G$.
    \item $c_{ci}(g_1,g_2,g_3)=c_{ci}(gg_1g^{-1},gg_2g^{-1},gg_3g^{-1})$ for any $g,g_1,g_2,g_3\in G$.
    \end{enumerate}

    We define the {\it conjugation-invariant total order} $<_{ci}$ to be a total order satisfying the condition:
    \[
        g_1<_{ci}g_2\Rightarrow gg_1g^{-1}<_{ci}gg_2g^{-1}\quad \textrm{for\,all\,}g,g_1,g_2\in G.
    \]
\end{definition}

We set $CiC(G)$ (resp. $Ci(G)$) to be the set containing all conjugation-invariant circular (resp. total) orders of $G$.

If $G$ is countable, then both types of conjugation-invariant orders have a dynamical realization. Namely, for any enumeration $\{g_i\}_{i\in \mathbb{N}}$ of $G$, the embedding $\iota$ is defined similarly to Lemma \ref{lem:const} and Lemma \ref{lem:const(total)}. In both cases, for $g,g'\in G$, we define $\rho(g')(\iota(g))=\iota(g'gg'^{-1})$. Then, we have the following dynamical realizations:
\begin{align}
    &CiC(G)\to \textrm{Homeo}_+(S^1)/\sim,\\
    &Ci(G)\to \textrm{Homeo}_+(\mathbb{R})/\sim.
\end{align}

\begin{proposition}
    The dynamical realization of a circular order of a conjugate quandle is the image of the dynamical realization of a conjugation-invariant circular order under the factor Conj. Namely, we have the following commutative diagram:
    \begin{align}
        \xymatrix{
            CiC(G)\ar[r]^-R\ar[d]^{\textrm{Conj}} & \textrm{Hom}(G,\textrm{Homeo}_+(S^1))\ar[d]^{\textrm{Conj}}\\
            RCQ(\textrm{Conj(G)})\ar[r]^-R & \textrm{Hom}(\textrm{Conj(G)},\textrm{Conj}(\textrm{Homeo}_+(S^1)))\ar@{}[lu]|{\circlearrowright}
        }
    \end{align}
    Moreover, both maps $\textrm{Conj}(\cdot)$ are bijections.
\end{proposition}
The above lemma holds in the case of total orders. By this proposition, the dynamical realization of a circular order of a conjugate quandle can be expressed as a group action.

\begin{proposition}
    It follows that $\textrm{Ci}(G)\supset \textrm{BO}(G)$. In particular, any circular or total order (not necessarily right or left invariant) is a conjugation-invariant circular or total order on any abelian group.
\end{proposition}

\section{Proof of Theorem\ref{thm:main2}}\label{sec:4}

We see how to restore the circular order from its dynamical realization.
\begin{definition}
    Let $Q$ be a circular orderable quandle which is semi-latin at $\hat{q}$ and $c$ one of its circular orders. We define the map $\lambda :Q^3\to \{-1,1\}$ as follows:
    \begin{align}
        \lambda(q,q',q'')=
        \begin{cases*}
            1 & if $c(\hat{q}*q,\hat{q}*q',\hat{q}*q'')=1$,\\
            -1 & if $c(\hat{q}*q,\hat{q}*q',\hat{q}*q'')=-1$
        \end{cases*}
    \end{align}
    for distinct elements $q,q',q''\in Q$. Note that the index $\lambda$ is well-defined since $Q$ is semi-latin at $\hat{q}$. We refer to $\lambda$ as the {\it index of $c$ at $\hat{q}$}. Furthermore, we denote $c$ as $c^{1}$ and the opposite order of $c$ as $c^{-1}$.
\end{definition}

\begin{remark}\label{rem:restore}
    Let $Q$ be a countable, circular orderable quandle which is semi-latin at $\hat{q}$, $c$ a circular order and $\rho$ its dynamical realization. The original order $c$ can be restored from $\rho$ via $\lambda$. For distinct elements $q,q',q''\in Q$, the following equivalences hold:
    \begin{align}
        c(q,q',q'')=\pm1&\Leftrightarrow c^{\lambda(q,q',q'')}(\hat{q}*q,\hat{q}*q',\hat{q}*q'')=\pm1\\
        &\Leftrightarrow \textrm{ord}^{\lambda(q,q',q'')}(\rho(q)(x_0),\rho(q')(x_0),\rho(q'')(x_0))=\pm1.
    \end{align}
    This indicates that $c$ can be restored from the single orbit at $\hat{q}$ if we first compute $\lambda$ from $c$. 
\end{remark}

\begin{remark}
    The original order $c$ can be restored without the index $\lambda$ in the following two cases.

    \begin{itemize}
        \item $c$ has a left-invariant element $\hat{q}$\footnote{Every left-invariant element is a semi-latin element.}. Then, the index $\lambda$ is constant.
        \item $Q$ is latin at $\hat{q}$. Then, the original order $c$ can be restored from $\rho$ as follows. By the latin property, for each $q,r\in Q$, there exists a bijection $l:Q^2\to Q$ satisfying the following condition:
        \[
            q*l(q,r)=r.
        \]
        We define a new map $c':Q^3\to\{-1,0,1\}$ as follows:
        \begin{align}
            c'(q,q',q'')=\textrm{ord}(\rho (l(\hat{q},q))(x_0),\rho (l(\hat{q},q'))(x_0),\rho (l(\hat{q},q''))(x_0)).
        \end{align}
        Then, we have:
        \begin{align}
            \textrm{ord}(\rho (l(\hat{q},q))(x_0),\rho (l(\hat{q},q'))(x_0),\rho (l(\hat{q},q''))(x_0))&=\textrm{ord}(\iota(q),\iota(q'),\iota(q''))\\
            &=c(q,q',q'').
        \end{align}
        Thus, we have $c'=c$.
    \end{itemize}
\end{remark}

We will prove Theorem \ref{thm:main2} by modifying the proof of \cite[Theorem 1.2]{MR3887426}, utilizing the index defined above.

\begin{lemma}\label{lem:same order}
    Let $G$ be a countable group, $A$ a finite set of $G$ and its Dehn quandle $\mathcal{D}({}^GA)$ right orderable and semi-latin at $\hat{g}\in \mathcal{D}({}^GA)$. Let $c_1$ be a circular order on $\mathcal{D}({}^GA)$ whose dynamical realization $\rho_1$ with basepoint $x_0$ can be extended to a group action of $G$ and $\lambda$ the index at $\hat{g}$ of $c_1$. Let $\rho_2\in \textrm{Hom}(\mathcal{D}({}^GA),\textrm{Conj}(\textrm{Homeo}_+(S^1)))$ be such that $x_0$ has a trivial stabilizer.

    The circular order induced by the orbit of $x_0$ under $\rho_2$ via $\lambda$ agrees with $c_1$ if and only if there exists a continuous, degree $1$ monotone map $f:S^1\to S^1$ such that $f(x_0)=x_0$ and $\rho_1(g)\circ f=f\circ \rho_2(g)$ for all $g\in \mathcal{D}({}^GA)$.
\end{lemma}

\begin{proof}
    Let $c_1$ be a circular order on $\mathcal{D}({}^GA)$, $\lambda$ its index at $\hat{g}$ and $\rho_1$ its dynamical realization with basepoint $x_0$.
    Let $\rho_2$ be an element of $\textrm{Hom}(\mathcal{D}({}^GA),\textrm{Conj}(\textrm{Homeo}_+(S^1)))$. Suppose that there is a continuous, degree $1$ monotone map $f:S^1\to S^1$ such that $f(x_0)=x_0$ and $\rho_1(g)\circ f=f\circ \rho_2(g)$ for all $g\in \mathcal{D}({}^GA)$. Then the cyclic order of the orbits of $x_0$ under $\rho_1(\mathcal{D}({}^GA))$ and $\rho_2(\mathcal{D}({}^GA))$ via $\lambda$ agree.

    For the converse, suppose that $\rho_2$ defines the same circular order as
    $\rho_1$ via $\lambda$. Note that, by assumption, $\rho_2$ can be extended to a group action of $G$. Define a map $f:\rho_2(\mathcal{D}({}^GA))(x_0)\to \rho_1(\mathcal{D}({}^GA))(x_0)$ by $\rho_2(g)(x_0) \mapsto \rho_1(g)(x_0)$. 
    Then for any $g,g'\in \mathcal{D}({}^GA)$, we have the following: 
    \begin{align}
        \rho_1(g')f(\rho_2(g)(x_0))&=\rho_1(g')\rho_1(g)(x_0)\\
        &=\rho_1(g'g)(x_0)\\
        &=f\rho_2(g'g)(x_0)\\
        &=f\rho_2(g')(\rho_2(g)(x_0)),
    \end{align}
    which means that $\rho_1(g')\circ f=f\circ\rho_2(g')$ on $\rho_2(\mathcal{D}({}^GA))(x_0)$. We show that $f$ can be extended to the circle.


    Suppose that $x$ is in the closure of the orbit of $x_0$ under $\rho_2(\mathcal{D}({}^GA))$. Then there exist $g_i$ in $\mathcal{D}({}^GA)$ such that $\rho_2(g_i)(x_0)\to x$; moreover we can choose these such that for each $i$, the triples $\rho_2(g_1)(x_0),\rho_2(g_i)(x_0),\rho_2(g_{i+1})(x_0)$ all have the same (positive or negative) orientation.

    Since the orbit of $x_0$ under $\rho_1$ has the same cyclic order as that of $\rho_2$, the triples $\rho_1(g_1)(x_0),\rho_1(g_i)(x_0),\rho_1(g_{i+1})(x_0)$ are also all positively oriented, so the sequence $\rho_1(g_i)(x_0)$ is monotone (increasing) in the closed interval obtained by cutting $S^1$ at $\rho_1(g_i)(x_0)$. Thus the sequence $\rho_1(g_i)(x_0)$ converges to some point, say $x$. Define $f(x)=y$. This is well-defined.
    
    If $\rho_2(\mathcal{D}({}^GA))(x_0)$ is dense in $S^1$, this completes the definition of $f$. If not, for each connected component of the complement of the closure of $\rho_2(\mathcal{D}({}^GA))(x_0)$, extend $f$ over the interval by defining it affinely. 
\end{proof} 

\begin{lemma}\label{lem:trichotomy}
    Let $Q$ be any subquandle of $\textrm{Conj}(\textrm{Homeo}_+(S^1))$. Then there
    are three mutually exclusive possibilities.
    \begin{enumerate}[label=(\arabic*)]
    \item There is a finite orbit. In this case, all finite orbits have the same
        cardinality.
    \item The action is minimal, i.e. all orbits are dense.
    \item There is a (unique) compact $Q$-invariant subset $K\subset S^1$, homeomorphic to a Cantor set, and contained in the closure of any orbit.
    \end{enumerate}
\end{lemma}

We refer to $K$ in case $(3)$ as an {\it exceptional minimal set}.

\begin{proof}
    The proof can be found for the case of groups by Ghys in \cite[Proposition 5.6]{MR1876932} and can be modified for that of quandles.
\end{proof}

\begin{lemma}\label{lem:conti}
    Let $G$ be a countable group, $A$ a finite subset of $G$ and its Dehn quandle $\mathcal{D}({}^GA)$ right orderable and semi-latin at $\hat{g}\in \mathcal{D}({}^GA)$. Let $c$ be a circular order on $\mathcal{D}({}^GA)$ whose dynamical realization $\rho$ based at $x_0$ can be extended to a group action of $G$. Let $U$ be any neighborhood of $\rho$ in $\textrm{Hom}(\mathcal{D}({}^GA), \textrm{Conj}(\textrm{Homeo}_+(S^1)))$. Then, there exists a neighborhood $V$ of $c$ in $RCQ(\mathcal{D}({}^GA))$ such that each order in $V$ has a dynamical realization based at $x_0$ contained in $U$.
\end{lemma}

\begin{proof}
    Fix an enumeration $\{q_1,q_2,\dots\}$ such that $q_1=\hat{q}$, and let $\rho :\mathcal{D}({}^GA)\to \textrm{Conj}( \textrm{Homeo}_+(S^1))$ be the dynamical realization of $c$ build with it and $\lambda$ an index of $c$. We may assume that a neighborhood $U=O_{(F,\varepsilon)}(\rho)\subset \textrm{Hom}(\mathcal{D}({}^GA),\textrm{Conj}(\textrm{Homeo}_+(S^1)))$ for some $\varepsilon$.
    
    Suppose first that $\rho(\mathcal{D}({}^GA))(x_0)$ is dense in $S^1$. Take a finite set $S\subset \mathcal{D}({}^GA)$ such that the neighborhoods $N(\rho(S)(x_0),(1/2)\varepsilon)$ is dense in $S^1$ for any $g\in F$, and let $n$ be such that $T=\{q_1,\dots,q_n\}$ contains $\hat{g}*S*F$. 
    In this way, if $c'$ is a circular order of $\mathcal{D}({}^GA)$ coinciding with $c$ over $T$, we can build a dynamical realization $\rho'$ of $c'$ using the enumeration $T$ that has the same base point as $\rho$. Then, for each $s\in S,g\in F$, we have
    \[
        \rho(g)\rho(s)(x_0)=\iota(\hat{q}*s*g)=\iota'(\hat{q}*s*g)=\rho'(g)\rho'(s)(x_0).
    \]
    Furthermore, $\rho'$ belongs to the $O_{(F,\varepsilon)}(\rho)$-neighborhood.

    In the case where $\rho(\mathcal{D}({}^GA))(x_0)$ is not dense, we claim that $x_0$ is an isolated point. This is due to the existence of an exceptional minimal set by Lemma \ref{lem:trichotomy}(see \cite[Lemma 3.12]{MR3887426} for details).

    Let $x_0$ be isolated. The action $\rho$ can be extended to a group action of $G$, which means that there exists $t\in \mathcal{D}({}^GA)$ such that the oriented interval $I=(x_0,\rho(t)(x_0))$ contains no other points of $\rho(G)(x_0)$.

    We use this observation to modify the construction from the proof of
    Proposition \ref{lem:conti} as follows. Given $\varepsilon$ and a finite set $F\subset \mathcal{D}({}^GA)$, let $S\subset \mathcal{D}({}^GA)$ be a finite set containing $F$ such that each interval $J$ in the complement of the set
    \[
        \bigcup_{s,s'\in S}\rho(s')\rho(s)(I)\cup \rho(s')^{-1}\rho(s)(I) 
    \]
    has length less than $\varepsilon$. Let $c'$ be a circular order that agrees with $c$ on $\hat{q}*S* S\cup \hat{q}*S*^{-1}S$. Then, as in the case where $\rho(\mathcal{D}({}^GA))(x_0)$ is dense, we may build $\rho'$ a dynamical realization of $c'$ that satisfies $\rho(s')^{\pm1}\rho(s)(x_0)=\rho'(s')^{\pm1}\rho(s)(x_0)$ for each $s,s'\in S$, i.e. $\rho(s')^{\pm1}\rho(s)(I)=\rho'(s')^{\pm1}\rho(s)(I)$. 

    For each $g\in S\cdot S\cup S \cdot S^{-1}$, let $h_g$ be the restriction of $\rho'(g)\rho^{-1}(g)$ to $\rho(g)(I)$. Note that this is a homeomorphism of $\rho(g)(I)$ fixing each endpoint. Let $h:S^1\to S^1$ be the homeomorphism defined by
    \begin{align}
        h(x)=
        \begin{cases*}
            h_g(x) & if $x\in \rho(g)(I)$ for some $g\in S\cdot S$,\\
            x & otherwise.
        \end{cases*}
    \end{align}
    Then $h\rho' h^{-1}$ is also a dynamical realization of $c'$. We now show
    that $h\rho' h^{-1}$ is in $O_{(S,\varepsilon)}(\rho)$. Let $s,s'\in S$. We have
    \[
        h\rho'(s)h^{-1}(\rho(s')(x_0))=h\rho'(s)\rho(s')(x_0)=\rho(s)\rho(s')(x_0).
    \]
    Moreover, for $t\in S$, if $x=\rho(t)(y)\in \rho(t)(I)$, then we have
    \begin{align}
        h\rho'(s)h^{-1}(x) &=h\rho'(s)\rho'(t)\rho(t)^{-1}\rho(t)(y)\\
        &=h\rho'(s)\rho'(t)(y)\\
        &=\rho(s)\rho(t)\rho'(t)^{-1}\rho'(s)^{-1}\rho'(s)\rho'(t)(y)\\
        &=\rho(s)(x)
    \end{align} 
    so $\rho(s)$ and $h\rho'(s)h^{-1}$ agree here. Finally, if $x$ is not in any such interval, then $\rho(s)(x)$, lies in the complement of the set $\bigcup_{s,s'\in S}\rho(s')\rho(s)(I)\cup \rho(s')^{-1}\rho(s)(I)$. Since both images $\rho(s)(x)$ and $h\rho'(s)h^{-1}(x)$ lie in the same complementary interval, which by construction has length less than $\varepsilon$, they differ by a distance less than $\varepsilon$.
\end{proof}

\begin{proof}[Proof of Theorem \ref{thm:main2}] 
    Let $c\in RCQ(\mathcal{D}({}^GA))$ be isolated, and let $\rho$ be its dynamical realization with basepoint $x_0$ and $\lambda$ be its index. Since $c$ is isolated, there exists a finite set $F\subset \mathcal{D}({}^GA)$ such that any order that agrees with $c$ on $F$ is equal to $c$. Since $\rho(F)(x_0)$ is a finite set, there exists a neighborhood $U$ of $\rho$ in $\textrm{Hom}(\mathcal{D}({}^GA),\textrm{Conj}(\textrm{Homeo}_+(S^1)))$ such that, for any $\rho'\in U$, the cyclic order of the set $\rho'(F)(x_0)$ via $\lambda$ agrees with that of $\rho(F)(x_0)$ via $\lambda$. Let $\rho'\in U$. If $x_0$ has a trivial stabilizer under $\rho'$, then the cyclic order on $\mathcal{D}({}^GA)$ induced by $\rho'(\mathcal{D}({}^GA))(x_0)$ via $\lambda$ agrees with $c$ on $F$, so is equal to $c$. By Proposition \ref{lem:same order}, this gives the existence of a map $h$ as in the theorem. If $x_0$ instead has a non-trivial stabilizer, say $K\subset \mathcal{D}({}^GA)$, then $\rho'(\mathcal{D}({}^GA))(x_0)$ gives a circular order on the set of cosets of $K$. Lemma \ref{lem:extend} below implies that this can be extended to an order on $\mathcal{D}({}^GA)$ in at least two different ways, both of which agree with $c$ on $F$. In particular, one of the order completions is not equal to $c$. This shows that $c$ is not an isolated point, contradicting our initial assumption. This completes the forward direction of the proof.

    For the converse, assume that $c$ is a circular order whose dynamical
    realization $\rho$ satisfies the rigidity condition in the theorem. Let $U$ be a
    neighborhood of $\rho$ so that each $\rho'$ in $U$ is semi-conjugate to $\rho$ by a continuous degree one monotone map $h$ as in the statement of the theorem.
    Lemma \ref{lem:conti} provides a neighborhood $V$ of $c$ such that any $c'\in V$ has a dynamical realization in $U$. Proposition \ref{lem:same order} now implies that the neighborhood $V$ consists of a single circular order, so $c$ is an isolated point.
\end{proof}

\begin{lemma}\label{lem:extend}
    Let $Q$ be a countable, circular orderable quandle with an action $\rho:Q\to \textrm{Conj}(\textrm{Homoe}(S^1))$ and $K$ a stabilizer $\textrm{Stab}(x)$. Then, each right order $<_K\in RQ(K)$ can be extended to a circular order of $Q$ which agrees with the cyclic order by the orbit $\rho(K\backslash Q)(x)$.
\end{lemma}

\begin{proof}
    Suppose that $<_K$ is a right order on $K$. Define $c(q_1,q_2,q_3)=1$ whenever $\textrm{ord}(\rho(q_1)(x),\rho(q_2)(x),\rho(q_3)(x))=1$. When $(q_1,q_2,q_3)$ forms a non-degenerate triple but $q_1(x)=q_2(x)\neq q_3(x)$, then we declare $c(q_1,q_2,q_3)=1$ whenever $q_1<_Kq_2$. This determines also the other cases where exactly two of the points $q_i(x)$ coincide. The remaining case is when $q_1(x)=q_2(x)=q_3(x)$, in which case we declare $c(q_1,q_2,q_3)$ to be the sign of the permutation $\sigma$ of $\{q_1,q_2,q_3\}$ such that $\sigma(q_1)<_K \sigma(q_2)<_K \sigma(q_3)$. This provides a well-defined circular order.
\end{proof}

We show the right order case of Theorem $\ref{thm:main2}$.

\begin{thm}\label{thm:main2.5}
    Let $G$ be a countable group, $A$ a subset of $G$ and the Dehn quandle $\mathcal{D}({}^GA)$ semi-latin at $\hat{q}\in \mathcal{D}({}^GA)$. Let $<$ be a right order of $\mathcal{D}({}^GA)$ whose dynamical realization $\rho$ can be extended to a group action of  $G$.
    
    Then, $<$ is isolated if and only if $\rho$ is rigid in the following sense: for every $\rho'$ sufficiently close to $\rho$ in $\textrm{Hom}(\mathcal{D}({}^GA),\textrm{Conj}(\textrm{Homeo}_+(\mathbb{R})))$ there exists a continuous, surjective monotone map $h:\mathbb{R}\to \mathbb{R}$ fixing the basepoint $x_0$ of the realization for $\hat{q}$, such that $\rho(q)\circ h=h\circ \rho'(q)$ for all $q\in \mathcal{D}({}^GA)$.
\end{thm}

\begin{proof}
    Since any right total order is also a right circular order, we can apply the analogous result of Lemma \ref{lem:conti} in the context of total orders. Thus, we can prove this in the same manner as Theorem \ref{thm:main2}.
\end{proof}

We confirm that the condition of Theorem \ref{thm:main2} can be changed to a group condition. This proves Corollary \ref{cor:main2}.

\begin{proposition}\label{prop:Dehn}
    Let $G$ be a countable group and $A$ a finite subset of $G$. The condition that the Dehn quandle $\mathcal{D}({}^GA)$ is semi-latin at some point is equivalent to the case satisfying
    \begin{enumerate}[label=\textbf{(\roman*)}]
    \item $G$ is a centerless group,
    \item $A$ does not contain the identity, and
    \item $A$ does not contain elements $g,h$ such that there exist distinct integers $i,j$ satisfying $g^i=h^j$ and $j\neq0$. 
    \end{enumerate}
\end{proposition}

\begin{proof}
    Suppose that $G$ has a non-trivial center; this contradicts Proposition \ref{prop:weak-latin}. Suppose that $A$ contains the identity. For any $g\in G\setminus\{1\}$, the relation $g*1=g*g$ holds, which leads to a contradiction. Next, suppose that there exist distinct elements $g,h\in A$ such that there exist $i\in \mathbb{Z}$ and $j\in \mathbb{Z}\setminus\{0\}$ satisfying $g^i=h^j$. Then, $g*h^j=g*g^i=g=g*g$ holds, which is again a contradiction.

    Conversely, fix $a_i\in A$ and $u\in G$. Suppose that for each $a_j,a_k\in A$ and $v,w\in G$, we have ${}^ua_i*{}^va_j={}^ua_i*{}^wa_k$. Then, this equation implies:
    \begin{align}
        {}^ua_i*{}^va_j&={}^ua_i*{}^wa_k\\
        {}^va_j{}^ua_i({}^va_j)^{-1}&={}^wa_k{}^ua_i({}^wa_k)^{-1}\\
        {}^ua_i(({}^va_j)^{-1}{}^wa_k)&=(({}^va_j)^{-1}{}^wa_k){}^ua_i.
    \end{align}
    By assumption (iii), ${}^ua_i$ does not commute with $({}^va_j^{-1}){}^wa_k$,
    which implies that $({}^va_j)^{-1}{}^wa_k=1$ by assumptions (i) and (ii). Therefore, we conclude that ${}^va_j={}^wa_k$.
\end{proof}

\begin{remark}
    This proof also shows that if a Dehn quandle is semi-latin at some point, then it is semi-latin at all points.
\end{remark}

\section{Proof of Theorem\ref{thm:main3}}\label{sec:5}

The proof is based on \cite{MR2859890}.

\begin{definition}
    Let $FQ(X)$ be a free quandle with a right order $<$. We define $B_X^n=\{q\in FQ(X)\mid \|q\|_X\leq n\}$ as the ball of radius $n$ in $FQ(X)$. Fixing $\hat{q}\in FQ(X)$, we let 
    \begin{align}
        &\mu_{n}^+(\hat{q})=\max_{<}\{\hat{q}*q\mid q\in B_X^n\},\\
        &\mu_{n}^-(\hat{q})=\min_{<}\{\hat{q}*q\mid q\in B_X^n\}.
    \end{align}
\end{definition}

\begin{proof}[Proof of Theorem \ref{thm:main3}]
    Let $X=\{x_1,\dots,x_n\}$ be a set and $FQ(X)$ be the free quandle generated by $X$. Note that $FQ(X)$ is semi-latin at every point $\hat{q}$. Let $<$ be a right order of $FQ(X)$ which has neither property E1 or E2, and $\lambda$ the index of $<$ at $\hat{q}$. Let $U_S$ be a neighborhood of $<$. We prove that there exists a right order $<'$ of $FQ(X)$ and elements $s,t\in FQ(X)$ such that:

    \begin{enumerate}[label=\textbf{(\roman*)}]
    \item $<'\in U_S$,
    \item $s<t$ while $t<'s$.
    \end{enumerate}

    Let $\rho$ be a dynamical realization of $<$ with the enumeration  satisfying $\iota(\hat{q})=x_0$ and $q_1=\hat{q}$. Let $F\subset B_{X}^m$. By the assumption of E1, one of the following cases holds.
    \begin{enumerate}[label=\textbf{(\arabic*)}]
    \item There exist distinct $M_1,M_2\in \mathbb{Z}_{\geq2}$ such that:
    \begin{align}
        \hat{q}<\hat{q}*\mu_m^+(\hat{q})<\hat{q}*\mu_{M_1}^+(\hat{q})<\hat{q}*\mu_{M_2}^+(\hat{q}).
    \end{align}
    \item There exist distinct $m_1,m_2\in \mathbb{Z}_{\geq2}$ such that:
    \begin{align}
        \hat{q}*\mu_{M_2}^-(\hat{q})<\hat{q}*\mu_{M_1}^-(\hat{q})<\hat{q}*\mu_m^-(\hat{q})<\hat{q}.
    \end{align}
    \end{enumerate}
    We only prove case (1), because case (2) can be proved in the same manner.

    By Joyce \cite{MR638121}, $FQ(X)$ is the union of all conjugation classes of elements of $X$: $\bigcup_{x\in X}{}^{F(X)}x$. Without loss of generality, we assume that $\mu_{M_1}^+(\hat{q})\in {}^{F(X)}x_1$ and $\mu_{M_1}^+(\hat{q})\in {}^{F(X)}x_2$ by the assumption of E2. There exist elements $a_0,a_1,b_0,b_1\in \mathbb{R}$ such that:
    \begin{align}
        x_0<_{\mathbb{R}}\rho(\mu_{m}^+(\hat{q}))(x_0)<_{\mathbb{R}}a_0<_{\mathbb{R}}a_1<_{\mathbb{R}}\rho(\mu_{M_1}^+(\hat{q}))(x_0)<_{\mathbb{R}}\rho(\mu_{M_2}^+(\hat{q}))(x_0)<_{\mathbb{R}}b_1<_{\mathbb{R}}b_0,
    \end{align}
    where $<_{\mathbb{R}}$ is the natural order on the line. We define $\varphi\in \textrm{Homeo}_+(\mathbb{R})$ satisfying $\textrm{Supp}(\varphi)=[a_0,b_0]$ and $\varphi(a_1)>_{\mathbb{R}}b_1$.  Then, we define the map $\rho_{\varphi}:FQ(X)\to \textrm{Conj}(\textrm{Homeo}_+(\mathbb{R}))$ as follows: 
    \begin{align}
        \rho_{\varphi}(q) = \begin{cases*}
            \varphi\circ\rho(q)  & if $q\in {}^{F(X)}x_1$,\\
            \rho(q) & otherwise.
        \end{cases*}
    \end{align}
    Indeed, this is a homomorphism by the freeness of the free quandle (see \cite{MR3981139}).
    
    We show that (i) holds. Let $g\in {}^{F(X)}x_{i}$. If $i\neq1$, then $\rho_{\varphi}(g)(x_0)=\rho(g)(x_0)$ holds. If $i=1$, then $\|g\|_{X} \leq m$ holds, which means that $\hat{q}*g<\hat{q}*\mu_{m}^+(\hat{q})$ holds. Then, we have
    \begin{align}
        \rho(g)(x_0)<_{\mathbb{R}}\rho(\mu_m)^+(\hat{q})(x_0)<_{\mathbb{R}}a_0,
    \end{align}
    which implies that $\rho_{\varphi}(g)(x_0)=\rho(g)\varphi(x_0)=\rho(x)(x_0)$ holds.

    We show that (ii) holds. We have
    \begin{align}
        \rho_{\varphi}(\mu_{M_1}^+(\hat{q}))(x_0)=\varphi\rho(\mu_{M_1}^+(\hat{q}))(x_0)>_{\mathbb{R}}\rho(\mu_{M_2}^+(\hat{q}))(x_0)=\rho_{\varphi}(\mu_{M_2}^+(\hat{q}))(x_0).
    \end{align}

    Therefore, the circular order $<'$ restored via $\lambda$ is contained in $U_S$ and does not equal to $<$.
\end{proof} 

The proof does not show whether there exist isolated right orders satisfying properties E1 or E2. However, we show that there exist right orders satisfying property E2 which are not isolated.

\begin{proposition}
    Let $<_{G}$ be an element of $\textrm{RO}(F(X))$. Then, a right order satisfying property E2 can be constructed from $<_{G}$ which is not isolated.
\end{proposition}

\begin{proof}
    For any reduced words $U,V\in F(X)$, there exists a factorization $U=uw$ and $V=vw$, where the suffix of $u$ does not equal one of $v$. We construct the total order $<_Q$ from $<_G$ as follows:
    \begin{align}
        {}^{uw}x_i<_Q{}^{vw}x_j \stackrel{{\rm def}}{\Longleftrightarrow}\begin{cases*}
            i<_{\mathbb{Z}}j,\\
            i=j \textrm{\, and \,} u<_{G}v,
        \end{cases*}
    \end{align}
    where $<_{\mathbb{Z}}$ is the natural order on $\mathbb{Z}$. The right-invariance of $<_G$ implies that $<_Q$ is a right order. For any ${}^{U'}x_i,{}^{V'}x_j\in FQ(X)$, we have $x_1<_Q{}^{U'}x_i<_Qx_n$, which implies ${}^{V'}x_j*x_1<_Q{}^{V'}x_j*{}^{U'}x_i<_Q{}^{V'}x_j*x_n$, i.e. property E2 is satisfied. It is known by McCleary \cite{MR787955} that $RO(F(X))$ does not have any isolated order. Thus, $<_Q$ is a non-isolated right order.
\end{proof}

\bibliographystyle{plain}
\bibliography{202410ref}

\begin{thebibliography}{10}

\bibitem{MR3664524}
Juan Alonso, Joaqu\'{\i}n Brum, and Crist\'{o}bal Rivas.
\newblock Orderings and flexibility of some subgroups of {$Homeo_+(\Bbb R)$}.
\newblock {\em J. Lond. Math. Soc. (2)}, 95(3):919--941, 2017.

\bibitem{ba2022circularorderabilityquandles}
Idrissa Ba and Mohamed Elhamdadi.
\newblock Circular orderability and quandles, 2022.

\bibitem{ba2023knot}
Idrissa Ba and Mohamed Elhamdadi.
\newblock Knot groups, quandle extensions and orderability, 2023.

\bibitem{MR4450681}
Valeriy~G. Bardakov, Inder Bir~S. Passi, and Mahender Singh.
\newblock Zero-divisors and idempotents in quandle rings.
\newblock {\em Osaka J. Math.}, 59(3):611--637, 2022.

\bibitem{MR3981139}
Valeriy~G. Bardakov, Mahender Singh, and Manpreet Singh.
\newblock Free quandles and knot quandles are residually finite.
\newblock {\em Proc. Amer. Math. Soc.}, 147(8):3621--3633, 2019.

\bibitem{MR4266358}
Danny Calegari and Lvzhou Chen.
\newblock Big mapping class groups and rigidity of the simple circle.
\newblock {\em Ergodic Theory Dynam. Systems}, 41(7):1961--1987, 2021.

\bibitem{MR1965363}
Danny Calegari and Nathan~M. Dunfield.
\newblock Laminations and groups of homeomorphisms of the circle.
\newblock {\em Invent. Math.}, 152(1):149--204, 2003.

\bibitem{MR2320157}
M.~A. Dabkowska, M.~K. Dabkowski, V.~S. Harizanov, J.~H. Przytycki, and M.~A.
  Veve.
\newblock Compactness of the space of left orders.
\newblock {\em J. Knot Theory Ramifications}, 16(3):257--266, 2007.

\bibitem{MR4669143}
Neeraj~K. Dhanwani, Hitesh~R. Raundal, and Mahender Singh.
\newblock Dehn quandles of groups and orientable surfaces.
\newblock {\em Fund. Math.}, 263(2):167--201, 2023.

\bibitem{MR4584679}
Serhii Dovhyi and Kyrylo Muliarchyk.
\newblock On the topology of the space of bi-orderings of a free group on two
  generators.
\newblock {\em Groups Geom. Dyn.}, 17(2):613--632, 2023.

\bibitem{MR1859702}
T.~V. Dubrovina and N.~I. Dubrovin.
\newblock On braid groups.
\newblock {\em Mat. Sb.}, 192(5):53--64, 2001.

\bibitem{MR1876932}
\'{E}tienne Ghys.
\newblock Groups acting on the circle.
\newblock {\em Enseign. Math. (2)}, 47(3-4):329--407, 2001.

\bibitem{MR3835743}
Trang Ha and Valentina Harizanov.
\newblock Orders on magmas and computability theory.
\newblock {\em J. Knot Theory Ramifications}, 27(7):1841001, 13, 2018.

\bibitem{MR638121}
David Joyce.
\newblock A classifying invariant of knots, the knot quandle.
\newblock {\em J. Pure Appl. Algebra}, 23(1):37--65, 1982.

\bibitem{MR2765563}
Peter~A. Linnell.
\newblock The space of left orders of a group is either finite or uncountable.
\newblock {\em Bull. Lond. Math. Soc.}, 43(1):200--202, 2011.

\bibitem{MR4033501}
Dominique Malicet, Kathryn Mann, Crist\'{o}bal Rivas, and Michele Triestino.
\newblock Ping-pong configurations and circular orders on free groups.
\newblock {\em Groups Geom. Dyn.}, 13(4):1195--1218, 2019.

\bibitem{MR3887426}
Kathryn Mann and Crist\'{o}bal Rivas.
\newblock Group orderings, dynamics, and rigidity.
\newblock {\em Ann. Inst. Fourier (Grenoble)}, 68(4):1399--1445, 2018.

\bibitem{MR4584769}
Kathryn Mann and Michele Triestino.
\newblock On the action of the {${\varSigma }(2,3,7)$} homology sphere group on
  its space of left-orders.
\newblock {\em Fund. Math.}, 261(3):297--302, 2023.

\bibitem{MR672410}
S.~V. Matveev.
\newblock Distributive groupoids in knot theory.
\newblock {\em Mat. Sb. (N.S.)}, 119(161)(1):78--88, 160, 1982.

\bibitem{MR787955}
Stephen~H. McCleary.
\newblock Free lattice-ordered groups represented as {$o$}-{$2$} transitive
  {$l$}-permutation groups.
\newblock {\em Trans. Amer. Math. Soc.}, 290(1):69--79, 1985.

\bibitem{MR2745552}
Andr\'{e}s Navas.
\newblock A remarkable family of left-ordered groups: central extensions of
  {H}ecke groups.
\newblock {\em J. Algebra}, 328:31--42, 2011.

\bibitem{MR3729413}
Takefumi Nosaka.
\newblock {\em Quandles and topological pairs}.
\newblock SpringerBriefs in Mathematics. Springer, Singapore, 2017.
\newblock Symmetry, knots, and cohomology.

\bibitem{MR4330281}
Hitesh Raundal, Mahender Singh, and Manpreet Singh.
\newblock Orderability of link quandles.
\newblock {\em Proc. Edinb. Math. Soc. (2)}, 64(3):620--649, 2021.

\bibitem{MR2859890}
Crist\'{o}bal Rivas.
\newblock Left-orderings on free products of groups.
\newblock {\em J. Algebra}, 350:318--329, 2012.

\bibitem{MR3460331}
Crist\'{o}bal Rivas and Romain Tessera.
\newblock On the space of left-orderings of virtually solvable groups.
\newblock {\em Groups Geom. Dyn.}, 10(1):65--90, 2016.

\bibitem{MR1805402}
Hamish Short and Bert Wiest.
\newblock Orderings of mapping class groups after {T}hurston.
\newblock {\em Enseign. Math. (2)}, 46(3-4):279--312, 2000.

\bibitem{MR2069015}
Adam~S. Sikora.
\newblock Topology on the spaces of orderings of groups.
\newblock {\em Bull. London Math. Soc.}, 36(4):519--526, 2004.

\bibitem{Takasaki1943AbstractionOS}
Mituhisa Takasaki.
\newblock Abstraction of symmetric transformations.
\newblock {\em Tohoku Mathematical Journal}, 49:145--207, 1943.

\end{thebibliography}

\end{document}